\documentclass{amsart}
\usepackage{amsmath}
\usepackage{graphicx}
\usepackage{amsfonts}
\usepackage{amssymb}
\usepackage{listings}
\usepackage{caption}%
\providecommand{\U}[1]{\protect\rule{.1in}{.1in}}
\makeatletter
\providecommand{\bigsqcap}{\mathop{\mathpalette\@updown\bigsqcup}}
\newcommand*{\@updown}
[2]{\rotatebox[origin=c]{180}{$\m@th#1#2$}}
\makeatother
\newtheorem{theorem}{Theorem}
\theoremstyle{plain}

\newtheorem{corollary}{Corollary}

\newtheorem{definition}{Definition}
\newtheorem{example}{Example}

\newtheorem{proposition}{Proposition}
\newtheorem{remark}{Remark}

\numberwithin{equation}{section}
\begin{document}
\title{Gerstenhaber-Schack Bialgebras}
\author{Ronald N. Umble}
\address{Department of Mathematics\\
Millersville University of Pennsylvania\\
Millersville, PA. 17551}
\email{ron.umble@millersville.edu}
\date{September 30, 2024}

\begin{abstract}
A \emph{Gerstenhaber-Schack (G-S) bialgebra} consists of a graded Hopf
algebra  $H$ together with multilinear operations $\omega^n_m\in
\{Hom^{-1}(  H^{\otimes m}, H^{\otimes n}):\linebreak m+n=4\},$ whose sum is 
the degree $-1$ component of a $2$-cocycle in the  G-S complex of $H$. 
A \emph{G-S extension} of a graded Hopf algebra $H$ is a  G-S bialgebra containing $H$. G-S extensions of $H$ are classified up to isomorphism by  the degree $-1$ component of the G-S cohomology group $H_{GS}^{2}(H;H)$. We  exhibit a space $X$ and a non-trivial
topologically induced G-S bialgebra structure on  $H^{\ast}\left(  \Omega X;\mathbb{Z}%
_{2}\right)  .$

\end{abstract}
\maketitle

\vspace*{-0.07in}



\section{Introduction}

A \emph{Gerstenhaber-Schack (G-S) bialgebra}  consists of a  graded Hopf algebra (gHa)
$H$ together with multilinear operations 
$\omega^n_m\in
\{Hom^{-1}(  H^{\otimes m}, H^{\otimes n}): m+n=4\},$  whose sum is
the degree $-1$ component of a  $2$-cocycle in the G-S 
complex of $H$ (antipodes are not assumed). 
A \emph{G-S extension} of a gHa $H$ is a G-S bialgebra containing $H$.  
G-S extensions of $H$ are classified up to isomorphism by the degree $-1$
component of the G-S cohomology group $H_{GS}^{2}(H;H)$.

Let $X$ be a  $\mathbb{Z}_{2}$-formal space. The bar construction $BA:=BH^{\ast}(X;\mathbb{Z}_{2})$ with standard differential and cofree
coproduct $\Delta_{BA}$ is a differential graded (dg) coalgebra model for the singular
cochains $S^{\ast}\left(  \Omega X;\mathbb{Z}_{2}\right).$  A homotopy Gerstenhaber algebra (hGa) structure on $H^{\ast}(X;\mathbb{Z}_{2})$ lifts to $BA$ and the induced product is Hopf compatible with $\Delta_{BA}$. Furthermore, under the right conditions, the dgHa
structure on $BA$ lifts to $H:=H^{\ast}(BA)  $ so that $H$ is a gHa model for $H^{\ast}\left(  \Omega X;\mathbb{Z}_{2}\right)
.$

When $H$ is free, there is a cocycle-selecting homomorphism $g:H\rightarrow
BA$ and an $A_{\infty}$-bialgebra structure $\omega$ on $H$ induced by
transferring the dgHa structure on $BA$ to $H$ along $g$. Since $H$
has  zero differential, $\omega$ specializes to a G-S bialgebra by forgetting all
operations  $\{\omega^n_m:m+n>4\}$ and all $A_{\infty}$-bialgebra structure
relations encoded by the biassociahedra $\{KK^n_m:m+n>5\}$ (see
Definitions  \ref{Defn1} and \ref{GSb}).

The article is organized as follows: Section 2 reviews the definition of an
$A_{\infty}$-bialgebra and defines $A_{k}$-bialgebras for $3\leq k < \infty$.
Section 3 reviews the definition of an $A_{\infty}$-bialgebra morphism and
defines morphisms of $A_{k}$-bialgebras for $3\leq k < \infty$. Section 4
reviews the G-S complex of a dgHa and presents our main result:  \medskip

\noindent\textbf{Theorem 1}. \emph{Given a gHa $(H,\mu,\Delta)$ and multilinear operations
	$\omega:=\{ \omega^1_3,\omega^2_2,\omega^3_1\}\subset Hom^{-1}(H^{\otimes m},H^{\otimes n})$, let $z:=\omega^1_3+\omega^2_2+\omega^3_1.$ Then}

\begin{enumerate}

\item[1.] $(H,\mu,\Delta,\omega)$\emph{ is a G-S extension if and only if 
$z$ is the degree $-1$ component of a  $2$-cocycle in the G-S 
complex of $H$.}\smallskip

\item[2.] \emph{G-S extensions $\omega$ and $\omega^{\prime}$
are equivalent if and only if $cls(z-z^{\prime})=0.$} 
\end{enumerate}
\vspace{.02in}

\noindent Section 5 reviews the Transfer Theorem and the relevant special case of its proof (the Transfer
Algorithm),  reviews the definition of a hGa, and exhibits a space $X$ with a
non-trivial  topologically induced G-S bialgebra structure on $H^*(BA)\approx H^{\ast}\left(  \Omega
X;\mathbb{Z} _{2}\right)  $.

\section{Biassociahedra and $A_{k}$-bialgebras}

In his 1963 seminal papers \textquotedblleft Homotopy associativity of
$H$-spaces I,II\textquotedblright\ \cite{Stasheff}, Jim Stasheff constructed
the associahedra $K:=\{K_{n}\}_{n\geq2}$ and used them to define $A_{n}$-algebras for
$2\leq n\leq\infty$. In \cite{SU1} and \cite{SU2}, S. Saneblidze and the
current author constructed the biassociahedra $KK:=\{KK^n_m\}_{m+n\geq3}$
and used them to define $A_{\infty}$-bialgebras; $A_{k}$-bialgebras for $3\leq k<\infty$ are
defined in  Definition \ref{Defn1} below.

The \emph{biassociahedron} $KK^n_m$ is a contractible $(m+n-3)$-dimensional
polytope, and $KK^n_1\cong KK^1_n$ is Stasheff's associahedron $K_{n}$.
The $2$-cell and edges of $KK^2_3$ pictured 
in Figure 1 are labeled by upward-directed graphs, each representing some composition 
of $\omega$-operations. In dimensions $\leq3$, the biassociahedra $KK$ constructed in \cite{SU2} agree 
with the polytopes under the same name and symbol constructed by M. Markl in \cite{Markl}.  \bigskip

\hspace{1in}\includegraphics[width=2.6in,height=2.5in]{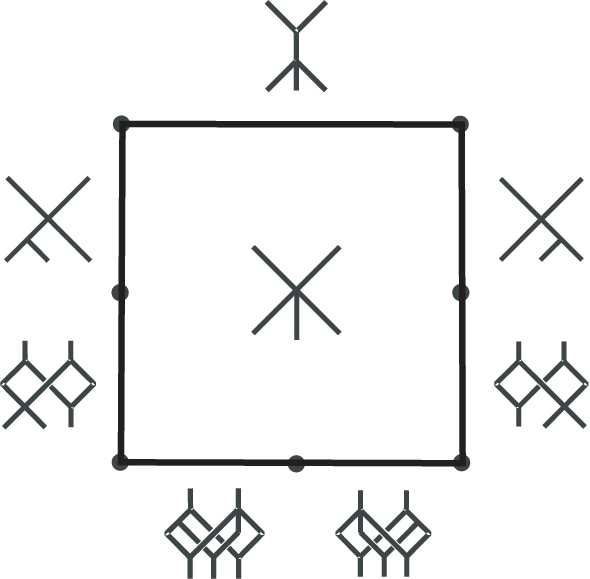}
\vspace{0.2in}

\begin{center}
Figure 1. The biassociahedron $KK^2_3$. 
\end{center}

\vspace{0.2in}

Let $R$ be a commutative ring with unity, let $\left(  A,d\right)  $ be a dg
$R$-module (dgm) with $|d|=+1$ and denote the tensor module of $A$ by $TA$.  The differential $\nabla$ on $Hom^{\ast}\left(  TA,TA\right)  $ induced by $d$ is
defined on $f\in Hom^{p}(A^{\otimes m},A^{\otimes n})$ by
\[
\nabla f:=d_{(n)} f-(-1)^{p} f\hspace{.03in}d_{(m)},
\]
where $d_{\left(  k\right)  }:=\sum_{s=0}^{k-1}\mathbf{1}^{\otimes s}\otimes
d\otimes\mathbf{1}^{\otimes k-s-1}$ is the linear extension. Denote the chain complex of cellular
chains on a polytope $P$ by $\left(  CC_{\ast}\left(  P\right)  ,\partial
\right)  $ and the top-dimensional cell of $KK^n_m$ by $\theta_{m}^{n}$.

\begin{definition}
\label{Defn1} Let $3\leq k\leq\infty$. An $A_{k}$\textbf{-bialgebra }consists
of a \textbf{dgm} $\left(  A,d\right)  $ together with \textbf{multilinear operations}
\[
\omega=\{ \omega^n_m\in Hom^{3-m-n}\left(  A^{\otimes m},A^{\otimes
n}\right)  : m+n\geq3 \},
\]
where $m+n\leq k$ when $k<\infty$, and \textbf{structure maps}
\[
\alpha=\{\alpha_m^n:(CC_{\ast}\left(  KK_{m,n}\right)  ,\partial)
\rightarrow\left(  Hom^{3-m-n}\left(  A^{\otimes m},A^{\otimes n}\right)
,\nabla\right)  \},
\]
where $\alpha_m^n$ is a chain map of matrads such that $\alpha_m^n
(\theta_{m}^{n})=\omega^n_m.$ The $KK^n_m$ \textbf{structure relation}  is
\begin{equation*}
\nabla\omega^n_m=(\nabla\circ \alpha_m^n) \theta_{m}^{n}  =(\alpha_m^n\circ\partial)\theta_{m}^{n}. 
\end{equation*}
An $A_k$-bialgebra $A$ is \textbf{strict }if $\nabla\omega_m^n=0$ for all $m$ and $n$. 
\end{definition}

Stasheff's $A_{n}$-algebras are $A_{n+1}$-bialgebras with $\omega^j_i=0$ for all $j>1$. 
Just as the operadic
structure of $K$ encodes the structure relations in $A_{n}$-algebras, the
matradic structure of $KK$ encodes the structure relations in $A_{k}$-bialgebras.

For notational simplicity denote $\mu:=\omega^1_2$ and
$\Delta:=\omega^2_1$, and denote the canonical permutation of
tensor factors by $\sigma_{m,n}:\left(A_1\otimes \cdots \otimes A_m\right)^{\otimes n}$ $\mapsto
A_1^{\otimes n}\otimes \cdots \otimes A_m^{\otimes n}$. The $KK^n_m$ structure relations
for $m+n\leq4$ are
\begin{equation}
\begin{tabular}
[c]{ccc}
$\nabla\mu=0$ & $\Leftrightarrow$ & $d\text{ is a derivation }$\\
$\nabla\Delta=0$ & $\Leftrightarrow$ & $d\text{ is a coderivation }$\\
$\nabla\omega^1_3=\mu\left(  \mu\otimes\mathbf{1}-\mathbf{1}\otimes\mu\right)$
& $\Leftrightarrow$ & $\mu\text{ is homotopy associative}$\\
$\nabla\omega^2_2=\left(  \mu\otimes\mu\right)  \sigma_{2,2}\left(
\Delta\otimes\Delta\right)  -\Delta\mu$ & $\Leftrightarrow$ & $\mu\text{ and }\Delta\text{ are homotopy compatible}$\\
$\nabla\omega^3_1=\left(  \mathbf{1}\otimes\Delta-\Delta\otimes
\mathbf{1}\right)  \Delta$ & $\Leftrightarrow$ & $\Delta\text{ is homotopy coassociative.}$
\end{tabular}
\label{Hopf}
\end{equation}
The $KK^n_m$ structure relations for $m+n=5$ are displayed in (\ref{KKnm}).

While strict $A_{4}$-bialgebras are gHa's by the relations in (\ref{Hopf}), the operations $\omega_m^n$, $m+n=4$, are unconstrained.  A ``Gerstenhaber-Schack 
bialgebra" is an $A_4$-bialgebra with zero differential and appropriately constrained operations $\omega_m^n$, $m+n=4$ (see Definition \ref{GSb}).

\section{Bimultiplihedra and Morphisms of $A_{k}$-Bialgebras}

In \cite{Stasheff}, J. Stasheff also introduced the multiplihedra 
$J:=\{J_{n}\}_{n\geq1}$ and used them to define
morphisms of $A_{n}$-algebras for $2\leq n \leq\infty$.  In \cite{SU2}, S.
Saneblidze and the current author introduced the bimultiplihedra
$JJ:=\{JJ^n_m\}_{m+n\geq2}$ and used them to define morphisms of $A_{\infty}
$-bialgebras; morphisms of $A_{k}$-bialgebras are defined in
Definition \ref{Defn2} below. The \emph{bimultiplihedron} $JJ^n_m$ is a
contractible $(m+n-2)$-dimensional polytope, and $JJ^n_1\cong JJ^1_n$ is
Stasheff's multiplihedron $J_{n}$.

Given dgm's $\left(  A,d_{A}\right)  $ and $\left(  B,d_{B}\right)  ,$ let
$\nabla$ denote the induced
differential on $Hom(TA,TB)$, and denote the top-dimensional cell of $JJ^n_m$ by $\mathfrak{f}
_{m}^{n}.$

\begin{definition}
\label{Defn2} Let $\left(  A,d_{A},\omega_{A}\right)  $ and $\left(
B,d_{B},\omega_{B}\right)  $ be $A_{k}$-bialgebras. A \textbf{morphism from}
$A$ \textbf{to} $B$ consists of multilinear maps
\[
G=\{g_{m}^{n}\in Hom^{2-m-n}(A^{\otimes m},B^{\otimes n}):m+n\geq2\},
\]
where $m+n\leq k$ when $k<\infty$, and \textbf{structure maps}
\[
\beta=\{\beta_m^n:\left(  CC_{\ast}\left(  JJ^n_m\right)  ,\partial
\right)  \rightarrow\left( Hom^{2-m-n}(A^{\otimes m},B^{\otimes n}),\nabla\right)  \},
\]
where $\beta_m^n$ is a chain map of relative matrads such that $\beta
_m^n\left(  \mathfrak{f}_{m}^{n}\right)  =g_{m}^{n}$. The $JJ^n_m
$-\textbf{structure relation} is
\begin{equation*}
\nabla g_{m}^{n}=(\nabla\circ \beta) \mathfrak{f}_{m}^{n} =(\beta\circ \partial)\mathfrak{f}_{m}^{n}.
\end{equation*}
Denote a morphism $G$ from $A$ to $B$ by $G:A\Rightarrow B$. A morphism
$\Phi=\{\phi_{m}^{n}\}:A\Rightarrow B$ is an \textbf{isomorphism} if $\phi
_{1}^{1}:A\rightarrow B{\ }$is an isomorphism of dgm's. 
\end{definition}

Stasheff's morphisms of $A_{n}$-algebras are morphisms of $A_{n+1}$-bialgebras
with $g_{i}^{j}=0$ for all $j>1$. Just as the
relative operadic structure of $J$ encodes the structure relations in a
morphism of $A_{n}$-algebras, the relative matradic structure of $JJ$ encodes
the structure relations in a morphism of $A_{k}$-bialgebras.

\begin{remark}
If $\Phi=\{\phi_{m}^{n}\}:A\Rightarrow A$ is an isomorphism, let $g=(\phi
_{1}^{1})^{-1}$ and define $\psi_{m}^{n}:=g^{\otimes n}\phi_{m}^{n};$ then
$\Psi=\{\psi_{m}^{n}\}:A\Rightarrow A$ is an isomorphism with $\psi_{1}
^{1}=\mathbf{1}_{A}$. Thus, when $\Phi:$ $A\Rightarrow A$ is an isomorphism,
we always assume that $\phi_{1}^{1}=\mathbf{1}_{A}.$ 
\end{remark}

To accommodate subscripts let $\omega^{n,m}:=\omega^n_m$, and for notational simplicity 
let $\mu_{X}:=\omega_{X}^{1,2}$ and $\Delta_{Y}:=\omega
_{Y}^{2,1}$. The $JJ^n_m$ structure relations for $2\leq m+n\leq4$ are
\medskip

$\nabla g_{1}^{1}=0$\hspace{1.06in} $\Leftrightarrow$\hspace{.1in} $g:=g_{1}^{1}$ is a chain map\medskip

$\nabla g_{2}^{1}=g\mu_{A}-\mu_{B} \left(  g\otimes g\right)  $\hspace{.1in}
$\Leftrightarrow$  \hspace{.1in}$g$ is homotopy multiplicative\medskip

$\nabla g_{1}^{2}=\Delta_{B}g-\left(  g\otimes g\right)  \Delta_{A}$
\hspace{.05in}$\Leftrightarrow$ \hspace{.1in}$g$ is homotopy comultiplicative\medskip

$\nabla g_{3}^{1}=\ g\omega_{A}^{1,3}\ -\mu_{B}\left(  g\otimes g_{2}
^{1}-g_{2}^{1}\otimes g\right)  \ +\ g_{2}^{1}\left(  \mu_{A}\otimes
\mathbf{1}-\mathbf{1}\otimes\mu_{A}\right)  \ -\ \omega_{B}^{1,3}g^{\otimes
3}$\medskip

$\nabla g_{2}^{2}=(g\otimes g)\omega_{A}^{2,2}-(\mu_{B}\otimes\mu_{B}
)\sigma_{2,2}(\Delta_{B}g\otimes g_{1}^{2}+g_{1}^{2}\otimes(g\otimes
g)\Delta_{A})+g_{1}^{2}\mu_{A}$\medskip

$\hspace{0.5in}-\left(  \mu_{B}(g\otimes g)\otimes g_{2}^{1}+g_{2}^{1}\otimes
g\mu_{A}\right)  \sigma_{2,2}(\Delta_{A}\otimes\Delta_{A})+\Delta_{B}g_{2}
^{1}-\omega_{B}^{2,2}(g\otimes g)\medskip$

$\nabla g_{1}^{3}=\ g^{\otimes3}\omega_{A}^{3,1}\ +\ \left(  g\otimes
g_{1}^{2}-g_{1}^{2}\otimes g\right)  \Delta_{A}+\left(  \mathbf{1}
\otimes\Delta_{B}-\Delta_{B}\otimes\mathbf{1}\right)  g_{1}^{2}\ -\omega
_{B}^{3,1}g.$

\section{The G-S Complex of a DG Hopf Algebra}

Let $\left(  H,d,\mu,\Delta\right)  $ be a dgHa with $\left\vert d\right\vert
=+1$ (when $\left\vert d\right\vert =-1$ the construction is completely
dual).  For $m\geq1,$ define left and right $H$-comodule actions $\lambda
_{m},\rho_{m}:H^{\otimes m}\rightarrow H^{\otimes m+1}$ by
\[
\begin{array}
[c]{l}%
\lambda_{1}=\rho_{1}:=\Delta\medskip\\
\lambda_{m}:=\Big(  \mu\left(  \mu\otimes\mathbf{1}\right)  \cdots\left(
\mu\otimes\mathbf{1}^{\otimes m-2}\right)  \otimes\mathbf{1}^{\otimes
m}\Big)  \sigma_{2,m}\Delta^{\otimes m}\medskip\\
\rho_{m}:=\Big(  \mathbf{1}^{\otimes m}\otimes\mu\left(  \mathbf{1}\otimes
\mu\right)  \cdots\left(  \mathbf{1}^{\otimes m-2}\otimes\mu\right)
\Big)  \sigma_{2,m}\Delta^{\otimes m}.
\end{array}
\]
For $n\geq1,$ define left and right $H$-module actions $\lambda^{n},\rho
^{n}:H^{\otimes n+1}\rightarrow H^{\otimes n}$ by
\[
\begin{array}
[c]{l}%
\lambda^{1}=\rho^{1}:=\mu\medskip\\
\lambda^{n}:=\mu^{\otimes n}\sigma_{n,2}\Big(  \left(  \Delta\otimes
\mathbf{1}^{\otimes n-2}\right)  \cdots\left(  \Delta\otimes\mathbf{1}\right)
\Delta\otimes\mathbf{1}^{\otimes n}\Big)\medskip\\
\rho^{n}:=\mu^{\otimes n}\sigma_{n,2}\Big(  \mathbf{1}^{\otimes n}
\otimes\left(  \mathbf{1}^{\otimes n-2}\otimes\Delta\right)  \cdots\left(
\mathbf{1}\otimes\Delta\right)  \Delta\Big) .
\end{array}
\]
Then $H^{\underline{\otimes}m}:=\left(  H^{\otimes m},\lambda_{m},\rho
_{m}\right)  $ is an $H$-bicomodule, $H^{\overline{\otimes}n}:=\left(
H^{\otimes n},\lambda^{n-1},\rho^{n-1}\right)  $ is an $H$-bimodule (when
$n=1$ the bimodule actions are undefined and $H^{\overline{\otimes}1}:=H),$
and $\{Hom^{p}(H^{\underline{\otimes}m},H^{\overline{\otimes}n}):p\in
\mathbb{Z}$ and $m,n\geq1\}$ is a trigraded $H$-bidimodule.

The linear extension 
\[
d_{\left(  k\right)  }:=\sum\limits_{s=0}^{k-1}\mathbf{1}^{\otimes s}\otimes
d\otimes\mathbf{1}^{\otimes k-s-1}
\]
and (co)bar differentials (forgetting shift of dimensions)
\[
\partial_{\left(  m\right)
}:=\sum\limits_{s=0}^{m-1}\left(  -1\right)  ^{s}\text{ }\mathbf{1}^{\otimes
s}\otimes\mu\otimes\mathbf{1}^{\otimes m-s-1}
\text{ and }\delta_{\left(  n\right)  }:=\sum\limits_{s=0}^{n-1}\left(
-1\right)  ^{s}\text{ }\mathbf{1}^{\otimes s}\otimes\Delta\otimes
\mathbf{1}^{\otimes n-s-1}
\]
induce strictly commuting differentials $\nabla,$ $\partial,$ and $\delta$ on
$\{Hom^{p}(H^{\underline{\otimes}m},H^{\overline{\otimes}n})\},$ which act on
an element $f$ of tridegree $\left(  p,m,n\right)  $ by
\begin{align*}
\nabla f  &  :=d_{\left(  n\right)  }f-\left(  -1\right)  ^{p}fd_{\left(
m\right)  }\\
\partial f  &  :=\lambda^{n}\left(  \mathbf{1}\otimes f\right)  -f\partial
_{\left(  m\right)  }-\left(  -1\right)  ^{m}\rho^{n}\left(  f\otimes
\mathbf{1}\right) \\
\delta f  &  :=\left(  \mathbf{1}\otimes f\right)  \lambda_{m}-\delta_{\left(
n\right)  }f-\left(  -1\right)  ^{n}\left(  f\otimes\mathbf{1}\right)
\rho_{m}.
\end{align*}
Note that $\nabla:\left(  p,m,n\right)  \mapsto\left(  p+1,m,n\right)  ,$
$\partial:\left(  p,m,n\right)  \mapsto\left(  p,m+1,n\right)  ,$ and
$\delta:\left(  p,m,n\right)  \mapsto\left(  p,m,n+1\right)  .$

The \emph{G-S complex of} $H$ is the triple complex $( Hom^{\ast
}(H^{\underline{\otimes}\ast},H^{\overline{\otimes}\ast}),\nabla
,\partial,\delta) .$ The subspace of total $r$-cochains in degree $p$ is
\[
C_{GS}^{r,p}\left(  H,H\right)  :=\bigoplus\limits_{p+m+n=r+1}Hom^{p}\left(
H^{\underline{\otimes}m},H^{\overline{\otimes}n}\right)
\]
and the total differential $D$ acts on a cochain $f$ of tridegree $\left(
p,m,n\right)  $ by
\[
Df: =\left(  -1\right)  ^{m+n}\nabla f+\partial f +\left(  -1\right)
^{m}\delta f,
\]
where the signs are chosen so that $D^{2}=0$ and the restriction of $D$ to
the  subspace $p=0$ agrees with the total differential on the G-S
double complex of an ungraded Hopf algebra \cite{GS}.

The subspace of total $r$-cocycles in degree $p$ is denoted by $Z_{GS}
^{r,p}\left(  H;H\right)  $. A general $2$-cocycle has components $\varphi^n_m$ of tridegree $\left(
p,m,n\right)  $ with $p+m+n=3$, and is an infinitesimal in the deformation 
theory of dgHa's \cite{Umble1}. A $2$-cocycle with $m+n\leq4$ is pictured in Figure 2. 
The $r^{th}$ G-S cohomology group in degree $p$ with coefficients in $H$ is
$H_{GS}^{r,p}\left(  H;H\right)  :=H^{\ast}\left(  C_{GS}^{r,p}\left(
H,H\right)  ,D\right)  . $ \vspace{0.2in}

\hspace{0.7in}\includegraphics[width=4.7in,height=2.32in]{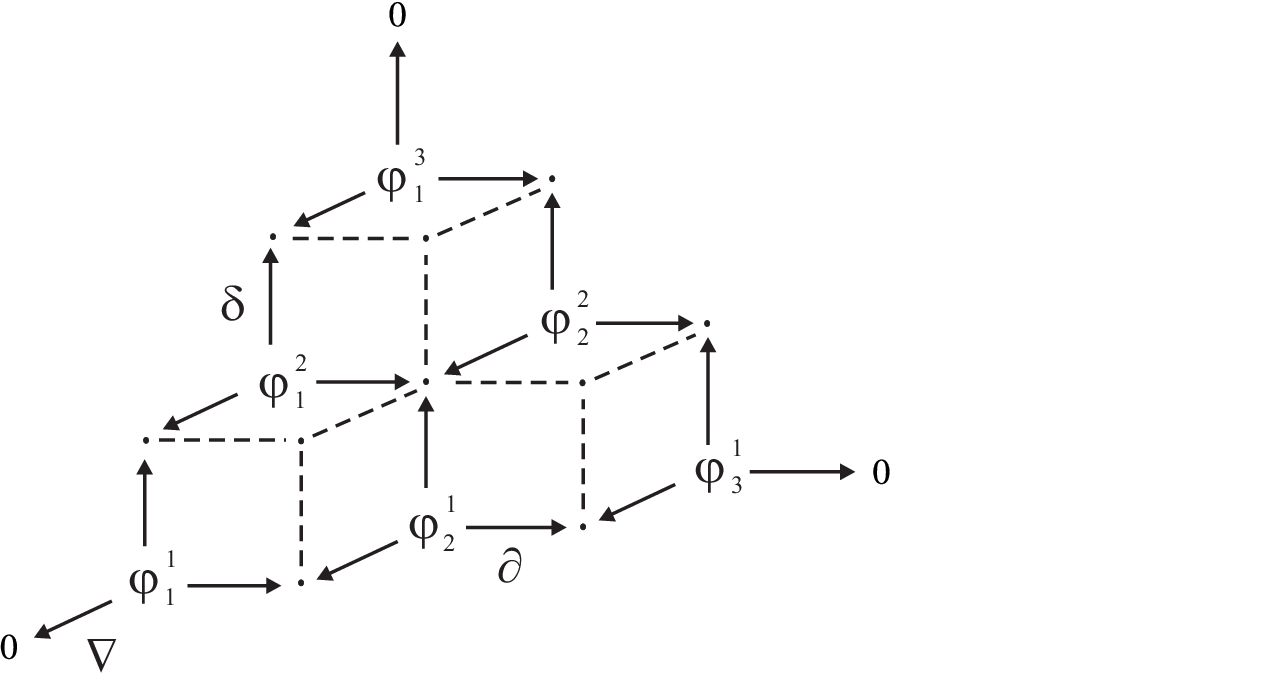}
\bigskip

\begin{center}
\vspace{-.1in}

Figure 2. A $2$-cocycle $\varphi^1_1+\varphi^1_2+\varphi^2_1+\varphi^1_3+\varphi^2_2+\varphi^3_1$ with components of tridegree $\left(  3-m-n,m,n\right)  $ and $m+n\leq4$. 
\end{center}
\vspace{0.2in}

It is truly remarkable that the $KK^n_m$ structure relations with $m+n=5$ and the $JJ^n_m$ structure relations with $m+n=4$ can be expressed in terms of G-S differentials.

\begin{example}
	To express the $KK^2_3$-structure relation in terms of G-S differentials, recall that $\alpha^n_m\left(\theta_{m}^{n}\right)=\omega^n_m$. Reading the graphical labels in Figure 1 from top-down and left-to-right, express each as a composition of $\omega$-operations. Then up to sign
	$$\nabla{\omega_{3} ^{2}}=\Delta \omega_{3}^{1} + \omega_{2}^{2}(\mu\otimes\mathbf{1}+\mathbf{1}\otimes\mu)
	+(\mu\otimes\mu)\sigma_{2,2}(\omega_{2}^{2}\otimes\Delta+\Delta\otimes\omega_{2}^{2})$$
	$$+\Big(\mu(\mu\otimes\mathbf{1})\otimes\omega_{3}
	^{1}+\omega_{3}^{1}\otimes\mu(\mathbf{1}\otimes\mu)\Big)\sigma_{2,3}\Delta^{\otimes3}.$$
	By definition,
	\[
	\partial\omega^2_2=\omega^2_2(\mu\otimes\mathbf{1}+\mathbf{1}\otimes
	\mu)+(\mu\otimes\mu)\sigma_{2,2}(\omega^2_2\otimes\Delta+\Delta\otimes\omega
	^2_2)\text{ and }
	\]
	\[
	\delta\omega^1_3=\Delta\omega^1_3+\Big(\mu(\mu\otimes\mathbf{1})
	\otimes\omega^1_3+\omega^1_3\otimes\mu(\mathbf{1}\otimes\mu)\Big)\sigma
	_{2,3}\Delta^{\otimes3}
	\]
	so that
	\[
	\nabla\omega^2_3 = \partial\omega^2_2 +\delta\omega^1_3.
	\]
	\vspace{.01in}
	
	\noindent Other $KK_m^n$ structure relations with $m+n=5$ can be expressed in a similar way. 
\end{example}

The $KK_m^n$ structure relations with $m+n=5$ are
	\begin{equation}
	\begin{tabular}
		[c]{cccc}
		$KK^1_4:$ & $\nabla\omega^1_4=\partial\omega^1_3$ & $\overset{\nabla= 0_{\mathstrut}}{\Rightarrow}$ & $\partial\omega^1_3=0$\bigskip\\
		$KK^2_3:$ & $\nabla\omega^2_3=\partial\omega^2_2-\delta\omega^1_3$ & $\Rightarrow$ & $\partial\omega^2_2-\delta\omega^1_3$=0\bigskip\\
		$KK^3_2:$ & $\nabla\omega^3_2=\partial\omega^3_1+\delta\omega^2_2$ & $\Rightarrow$ & $\partial\omega^3_1+\delta\omega^2_2=0$\bigskip\\
		$KK^4_1:$ & $\nabla\omega^4_1=-\delta\omega^3_1$ & $\Rightarrow$ & $\delta\omega^3_1=0$.\bigskip
	\end{tabular}
	\label{KKnm}
\end{equation}
The strict relations in (\ref{KKnm}) provide the linkage we need to form the degree $-1$ component $\omega^1_3+\omega^2_2+\omega^3_1$ of a strict $2$-cocycle (see Figure 3).
\vspace{.2in}

\begin{center}
	$
\begin{array}
	[c]{ccccccc}%
	\delta\omega_{1}^{3}=0 &  &  &  &  &  &\vspace{-.1in} \\
	&  &  &  &  &  & \\
	\uparrow &  &  &  &  &  & \vspace{-.1in}\\
	&  &  &  &  &  & \\
	\omega_{1}^{3} &\hspace{-.15in} \longrightarrow & \partial\omega_{1}^{3}+\delta\omega_{2}^{2}=0
	&  &  &  & \vspace{-.1in}\\
	&  &  &  &  &  & \\
	&  & \uparrow &  &  &  & \vspace{-.1in}\\
	&  &  &  &  &  & \\
	&  & \omega_{2}^{2} & \hspace{-.3in}\longrightarrow & \partial\omega_{2}^{2}-\delta
	\omega_{3}^{1}=0 &  & \vspace{-.1in}\\
	&  &  &  &  &  & \\
	&  &  &  & \uparrow &  & \vspace{-.1in}\\
	&  &  &  &  &  & \\
	&  &  &  & \omega_{3}^{1} & \hspace{-.3in}\longrightarrow & \partial\omega_{3}^{1}=0
\end{array}
$
\vspace{.2in}

	Figure 3. The degree -1 component of a strict $2$-cocycle.
\end{center}
\vspace{.2in}

Similarly, the $JJ^n_m$ structure relations for an isomorphism $\Phi:\left(  H,d,\mu,\Delta,\omega_{A}\right)  \Rightarrow
\left(  H,d,\mu,\Delta,\omega_{B}\right)  $ of $A_{4}$-bialgebras are 
\begin{equation}
	\begin{tabular}
		[c]{cc}
		$JJ_{1,3}:$ & $\nabla\phi_{3}^{1}=\omega_{A}^{1,3}-\partial\phi_{2}^{1}
		-\omega_{B}^{1,3}$\\
		$JJ_{2,2}:$ & $\nabla\phi_{2}^{2}=\omega_{A}^{2,2}-\partial\phi_{1}^{2}
		-\delta\phi_{2}^{1}-\omega_{B}^{2,2}$\\
		$JJ_{3,1}:$ & $\nabla\phi_{1}^{3}=\omega_{A}^{3,1}+\delta\phi_{1}^{2}
		-\omega_{B}^{3,1}.$%
	\end{tabular}
	\ \label{JJnm-6}
\end{equation}

Indeed, the algebraic representations of the $2$-dimensional biassociahedra
and bimultiplihedra displayed in (\ref{KKnm}) and (\ref{JJnm-6}) appear quite naturally and
were hiding in the G-S complex more than a decade before the corresponding
polytopes appeared in [7].

The differentials $\nabla,$ $\partial,$ and $\delta$,  respectively defined in terms of
$d,\ \mu,\ \text{and}\ \Delta$, capture the interactions of a higher order
operation with the underlying dgHa structure but completely miss its
interactions with the higher order structure. Consequently, the
$KK^n_m$ structure relations cannot be expressed in terms of the differentials
 in the G-S complex when $m+n\geq6$.

Now by definition, an $A_4$-bialgebra $(H,\mu,\Delta,\omega^1_3,\omega^2_2,\omega^3_1)$ (with zero differential) 
is a gHa with three higher order operations of degree $-1$.  By homogeneity, $D(\omega^1_3+\omega^2_2+\omega^3_1)=0$ 
if and only if $\delta\omega^3_1=\partial\omega^2_2-\delta\omega^1_3=\partial\omega^3_1
+\delta\omega^2_2=\delta\omega^3_1=0$. Thus we define:

\begin{definition}
\label{GSb} An $A_{4}$-bialgebra $(H,\mu,\Delta,\omega^1_3,\omega^2_2,\omega^3_1)$ is a \textbf{Gerstenhaber-Schack bialgebra} if
\begin{equation}
D(\omega^1_3+\omega^2_2+\omega^3_1)=0. \label{cocycle}
\end{equation}
A \textbf{G-S extension} of a gHa $(H,\mu,\Delta)$
is a G-S bialgebra of the form $(H,\mu,\Delta,\omega:=\{\omega^1_3,\omega^2_2,\omega^3_1\})$; we sometimes refer to
$\omega$ as a G-S extension when the context is clear. G-S extensions
$\omega$ and $\omega^{\prime}$ are \textbf{equivalent }if there exists an
isomorphism $\Phi:\left(  H,\mu,\Delta,\omega\right)  \Rightarrow\left(
H,\mu,\Delta,\omega^{\prime}\right)  $ of $A_{4}$-bialgebras. A G-S extension
$\omega$ is \textbf{trivial }if $\left(  H,\mu,\Delta,\omega\right)
\cong\left(  H,\mu,\Delta\right)  $. 
\end{definition}


\begin{theorem}
\label{four}Given a gHa $(H,\mu,\Delta)$ and multilinear operations
$\omega:=\{ \omega^1_3,\omega^2_2,\omega^3_1\}\subset Hom^{-1}(H^{\otimes m},H^{\otimes n})$, let $z:=\omega^1_3+\omega^2_2+\omega^3_1.$ Then

\begin{enumerate}

\item[1.] $\omega$ is a G-S extension if and only if $D(z)=0 $.\smallskip

\item[2.] G-S extensions $\omega$ and $\omega^{\prime}$ are equivalent if
and only if $cls(z-z^{\prime})=0.$ 
\end{enumerate}
\end{theorem}

\noindent\textbf{Proof.} The proof of Part 1 is trivial.

\noindent Proof of part 2: $\omega\sim\omega^{\prime}$ if and only if there exists an
isomorphism $\Phi=\{\mathbf{1}_{H},\phi_{m}^{n}:n+m=3,4\}:(H,\mu,\Delta
,\omega)\Rightarrow(H,\mu,\Delta,\omega)$ of $A_{4}$-bialgebras if and only
if  $\Phi$ satisfies the $JJ^n_m$-structure relations for $m+n=3,4$, which
hold  trivially when $m+n=3$. Since $\nabla=0$, the $JJ^n_m$-structure
relations  in (\ref{JJnm-6}) reduce to
\[
\partial\phi_{2}^{1}=\omega^1_3-(\omega^{\prime})^1_3
\]
\begin{equation}
\partial\phi_{1}^{2}+\delta\phi_{2}^{1}=\omega^2_2-(\omega^{\prime} )^2_2
\label{JJ-reduced}%
\end{equation}
\[
-\delta\phi_{1}^{2}=\omega^3_1-(\omega^{\prime})^3_1.
\]
Therefore $\omega\sim\omega^{\prime}$ if and only if there exists a $\left(
1,-1\right)  $-cochain $\phi_{2}^{1}+\phi_{1}^{2}$ such that the structure
relations in (\ref{JJ-reduced}) hold if and only if
\[
D\left(  \phi_{2}^{1}+\phi_{1}^{2}\right)  =\partial\phi_{2}^{1}+\left(
\partial\phi_{1}^{2}+\delta\phi_{2}^{1}\right)  -\delta\phi_{1}^{2}
=\omega-\omega^{\prime}.
\]
$\ \square$

\begin{corollary}
A G-S extension $\omega$ is trivial if and only if
$cls(z)=0$. 
\end{corollary}

\noindent\textbf{Proof.} Set $\omega^{\prime}=0$ and apply Theorem 1, Part 2. $\ \square$

\begin{corollary}
G-S extensions of a gHa $H$ are parametrized by $Z_{GS}
^{2,-1}\left(  H;H\right)  $ and classified up to isomorphism by
$H_{GS}^{2,-1}\left(  H;H\right)  .$ 
\end{corollary}

\begin{example}
\label{Ex1}Consider the $\mathbb{Z}_{2}$-dg algebra (dga)
\[
A=\left\langle 1,a_{2},a_{3},b_{3},a_{2}a_{3}=a_{3}a_{2}\right\rangle ,
\]
where $\left\vert x_{i}\right\vert =i,$ and the bar construction $BA$ with
standard differential $d_{BA}$, shuffle product $sh$, and cofree coproduct $\Delta_{BA}$. 
Denote a homogeneous element 
$\downarrow x_{1}\otimes \cdots\otimes\downarrow x_{n}\in BA$ by $[x_{1}|\cdots|x_{n}]$. 
Then $BA$  is a dgHa such
that $d_{BA}\left(  [a_{2}|a_{3}]+[a_{3}|a_{2}]\right)  =0,$ and $H_0:=H^{\ast
}\left(  BA\right)  $ is a gHa with induced product $\mu$ and
coproduct $\Delta.$ Let
$\alpha_{i}:=cls[a_{i+1}],\ \beta:=cls[b_{3}]$, and $\gamma:=\mu\left(
\alpha_{1}\otimes\alpha_{2}\right)  =cls([a_{2}|a_{3}]+[a_{3}|a_{2}]).$ Then
$\mu$ acts as the shuffle product except
\[
\mu\left(  \alpha_{i}\otimes\gamma\right)  =\mu\left(  \gamma\otimes\alpha
_{i}\right)  =0
\]
(by associativity) and $\Delta$ acts as the free coproduct except
\[
\Delta\gamma=1\otimes\gamma+\alpha_{1}\otimes\alpha_{2}+\alpha_{2}
\otimes\alpha_{1}+\gamma\otimes1
\]
(by Hopf compatibility). Define $\phi_{2}^{1},$ $\omega^1_3,$ and
$\omega^2_2$ to be zero except\smallskip%
\[
\phi_{2}^{1}\left(  \beta\otimes\beta\right)  :=\gamma,\text{ \ \ }
\omega^2_2\left(  \beta\otimes\beta\right)  :=\alpha_{1}\otimes\alpha
_{2}+\alpha_{2}\otimes\alpha_{1},\text{\ \ \ and }
\]
\[
\omega^1_3\left(  \beta\otimes\beta\otimes\beta\right)  :=\mu\left(
\beta\otimes\gamma\right)  \text{.}
\]
By direct calculation,
\[
\left(  \partial\phi_{2}^{1}\right)  \left(  \beta\otimes\beta\otimes
\beta\right)  =\mu\left(  \beta\otimes\gamma\right)  =\omega^1_3\left(
\beta\otimes\beta\otimes\beta\right)  \text{ \ and}
\]
\begin{align*}
\left(  \delta\phi_{2}^{1}\right)  \left(  \beta\otimes\beta\right)   &
=\left(  \left(  \mu\otimes\psi_{2}^{1}+\psi_{2}^{1}\otimes\mu\right)
\sigma_{2,2}\left(  \Delta\otimes\Delta\right)  +\Delta\psi_{2}^{1}\right)
\left(  \beta\otimes\beta\right) \\
&  =\alpha_{1}\otimes\alpha_{2}+\alpha_{2}\otimes\alpha_{1}=\omega
^2_2\left(  \beta\otimes\beta\right)  .
\end{align*}
Therefore
\[
D \phi_{2}^{1} =\partial\phi_{2}^{1} +\delta\phi_{2}^{1} =\omega^1_3
+\omega^2_2.
\]
Since $cls\left(  \omega^1_3+\omega^2_2\right)  =0$, the G-S extension
$\tilde{H}:=\left(  H,\mu,\Delta,\omega^1_3,\omega^2_2\right)  $ is
trivial by Theorem \ref{four}, Part 2, and indeed, $\Phi=\left\{
\mathbf{1}_{A},\psi_{2}^{1}\right\}  :\tilde{H}\Rightarrow H_{0}$ is an
isomorphism of $A_{4}$-bialgebras. 
\end{example}

The remainder of this article considers an induced $A_{\infty}$-bialgebra
structure $\omega$ on a particular loop cohomology $H$ and applies Theorem
\ref{four} to obtain a non-trivial G-S extension of the underlying gHa structure.

\section{A Topological Application}

\subsection{The Transfer Theorem and Algorithm}

Let $X$ be a space. Under mild conditions, the Transfer Algorithm induces a
canonical $A_{\infty}$-bialgebra structure on $A:=H^*(BA)\approx H^{\ast}\left(  \Omega
X;\mathbb{Z}_2\right)  .$\ We state the Transfer Theorem when $A$ is free; the Transfer Algorithm appears in the proof. For the
general case and a proof of uniqueness, which is omitted, see \cite{SU2}.

\begin{theorem}
[\textbf{The Transfer Theorem}]\label{transfer}Let $\left(  A,d_{A}\right)  $
be a free dgm, let $(B,$ $d_{B},\omega_{B})$ be an $A_{\infty} $-bialgebra,
and let $g:A\rightarrow B$ be a chain map/homology isomorphism. Then $g$
induces an $A_{\infty}$-bialgebra structure $\omega_{A}=\{\omega_{A}^{n,m}\}$
on $A$ and extends to a map $G=\{g_{m}^{n}: g_{1} ^{1}=g\}:A\Rightarrow
B$\ of  $A_{\infty}$-bialgebras. Furthermore, $\omega_{A}$ and $G$ are unique
up to isomorphism. 
\end{theorem}

\noindent\textbf{Proof} \textbf{(The Transfer Algorithm).} For $f\in Hom(A^{\otimes m},A^{\otimes n})$ define $\tilde{g}(f):=g^{\otimes n}f$ and note that
$\tilde{g}$ is a homology isomorphism since $A$ is free. We obtain an induced $A_\infty$-bialgebra  structure by simultaneously constructing a chain map $\alpha
_{A}:CC_{\ast}\left(  KK\right)  \rightarrow Hom(TA,TA)$ of matrads and a chain map
$\beta:CC_{\ast}\left(  JJ\right)  \rightarrow Hom(TA,TB)$ of relative matrads.

Thinking of $JJ^n_m$ as a subdivision of the cylinder $KK^n_m\times I,$
denote the top dimensional cells of $KK^n_m$ and $JJ^n_m$ by $\theta
_{m}^{n}$ and $\mathfrak{f}_{m}^{n}$, and identify the faces $KK^n_m%
\times0$  and $KK^n_m\times1$ of $JJ^n_m$ with $\theta_{m}^{n}\left(
\mathfrak{f} _{1}^{1}\right)  ^{\otimes m}$ and $\left(  \mathfrak{f}_{1}%
^{1}\right)   ^{\otimes n}\theta_{m}^{n},$ respectively. By hypothesis, there
is a map of  matrads $\alpha_{B}:CC_{\ast}(KK)\rightarrow(U_{B},\nabla)$ such
that  $\alpha_{B}(\theta_{m}^{n})=\omega_{B}^{n,m}.$

To initialize the induction, define $\mathcal{\beta}:CC_{\ast}\left(
JJ^1_1\right)  \rightarrow Hom^{0}\left(  A,B\right)  $ by $\beta\left(
\mathfrak{f}_{1}^{1}\right)  =g_{1}^{1}=g$, and extend $\mathcal{\beta}$ to
$CC_{\ast}\left(  JJ^1_2\right)  \rightarrow Hom^{-1}\left(  A\otimes A
,B\right)  $ and $CC_{\ast}\left(  JJ^2_1\right)  \rightarrow
Hom^{-1}(  A,$ $ B\otimes B)  $ 
in the following way: On the vertices
$\theta_{2}^{1}\left(  \mathfrak{f}_{1}^{1}\otimes\mathfrak{f}_{1}^{1}\right)
\in JJ^1_2$ and $\theta_{1}^{2}\mathfrak{f}_{1}^{1}\in JJ^2_1$, define
$\beta\left(  \theta_{2}^{1}\left(  \mathfrak{f}_{1}^{1}\otimes\mathfrak{f}
_{1}^{1}\right)  \right)  =\mu_{B}\left(  g\otimes g\right)  $ and
$\beta\left(  \theta_{1}^{2}\,\mathfrak{f}_{1}^{1}\right)  =\Delta_{B}g.$
Since $\mu_{B}\left(  g\otimes g\right)  $ and $\Delta_{B}g$
are $\nabla$-cocycles, and $\tilde{g}_{\ast}$ is an isomorphism, there exist
cocycles $\mu_{A}\in Hom^0(A\otimes A,A)$ and $\Delta_{A}\in Hom^0(A,A\otimes A)$ such that
$
\tilde{g}_{\ast}[\mu_{A}]=[\mu_{B}\left(  g\otimes g\right)
]\text{ \ and \ }\tilde{g}_{\ast}[\Delta_{A}]=[\Delta_{B}g]\,.
$
Thus $\left[  g\mu_{A}-\mu_{B}\left(  g\otimes g\right)
\right]  =\left[  \Delta_{B}g-\left(  g\otimes g\right)  \Delta_{A}\right]  =0,$ 
and there exist cochains $g_{2}^{1}\in Hom^{-1}(A,B\otimes B)$ and $g_{1}^{2} \in Hom^{-1}(A\otimes A,B)$ such that $\nabla g_{2}^{1}=g\mu_{A}-\mu_{B}\left(  g\otimes g\right)$
 and $\nabla g_{1}^{2}=\Delta_{B}g-\left(  g\otimes g\right)
\Delta_{A}\,.$ 

For $m+n=3$, define $\alpha_{A}:CC_{\ast}\left(  KK^n_m\right)
\rightarrow Hom^0(  A^{\otimes m},$ $A^{\otimes n})  $
by $\alpha
_{A}(\theta_{m}^{n}):=\omega_{A}^{n,m}$ and $\beta:CC_{\ast}\left(
JJ^n_m\right)  \rightarrow Hom^*\left(  A^{\otimes m},B^{\otimes n}\right)  $
by
\[
\begin{array}
[c]{rllll}
\beta(\mathfrak{f}_{m}^{n}) & := & g_{m}^{n}\in Hom^{-1}(A^{\otimes m},A^{\otimes n})\vspace{1mm}   \\
\beta(\mathfrak{f}_{1}^{1}\,\theta_{2}^{1}) & := & g\,\mu_{A}\in Hom^{0}(A\otimes A,A) & 
\vspace{1mm}  \\
\beta\big((\mathfrak{f}_{1}^{1}\otimes\mathfrak{f}_{1}^{1})\,\theta_{1}^{2}\big) & := &
(g\otimes g)\,\Delta_{A}\in Hom^{0}(A,A\otimes A)\,.  
\end{array}
\]

Inductively, given $m+n\geq 4,$ assume that for
$i+j<m+n$ there exists a map of matrads $\alpha_{A}:CC_{\ast}(
KK^j_i)  \rightarrow Hom^{3-i-j}\left(  A^{\otimes i},A^{\otimes j}\right)  $
and a map of relative matrads $\beta:CC_{\ast}(  JJ^j_i)
\rightarrow Hom^{2-i-j}(  A^{\otimes i},B^{\otimes j})  $ such that
$\alpha_{A}(\theta_{i}^{j})=\omega_{A}^{j,i}$ and $\beta(\mathfrak{f}_{i}
^{j})=g_{i}^{j}$. Thus we are given chain maps $\alpha_{A}:CC_{\ast}\left(
\partial KK^n_m\right)  \rightarrow Hom^{4-m-n}\left(  A^{\otimes m},A^{\otimes
n}\right)  $ and $\beta:CC_{\ast}\left(  \partial JJ^n_m\smallsetminus
\operatorname*{int}KK^n_m\times1\right)  \rightarrow Hom^{3-m-n}\left(  A^{\otimes
m},B^{\otimes n}\right)  .$ We wish to extend $\alpha_{A}$ to the top cell
$\theta_{m}^{n}$ of $KK^n_m$, and $\beta$ to the codimension 1 cell $\left(
\mathfrak{f}_{1}^{1}\right)  ^{\otimes n}\theta_{m}^{n}$ and the top cell
$\mathfrak{f}_{m}^{n}$ of $JJ^n_m$. Since $\alpha_{A}$ is a map of matrads,
the components of the cocycle
\[
z=\alpha_{A}\big(  CC_{\ast}(\partial KK^n_m)\big)  \in Hom^{4-m-n}\left(
A^{\otimes m},A^{\otimes n}\right)
\]
are expressed in terms of $\omega_{A}^{j,i}$ with $i+j<m+n;$ similarly, since
$\beta$ is a map of relative matrads, the components of the cochain
\[
\varphi=\beta\big(  CC_{\ast}(\partial JJ^n_m\smallsetminus
\operatorname*{int}KK^n_m\times1)\big)  \in Hom^{3-m-n}\left(  A^{\otimes
m},B^{\otimes n}\right)
\]
are expressed in terms of $\omega_{B},$ $\omega_{A}^{j,i}$ and $g_{i}^{j}$
with $i+j<m+n.$ Clearly $\tilde{g}\left(  z\right)  =\nabla\varphi;$ and
$\left[  z\right]  =\left[  0\right]  $ since $\tilde{g}$ is a homology
isomorphism. Now choose a cochain $b\in Hom^{3-m-n}\left(  A^{\otimes
m},A^{\otimes n}\right)  $ such that $\nabla b=z;$ then
$
\nabla\left(  \tilde{g}\left(  b\right)  -\varphi\right)  =\nabla\tilde
{g}\left(  b\right)  -\tilde{g}\left(  z\right)  =0.
$
Choose a class representative $u\in\tilde{g}_{\ast}^{-1}\left[  \tilde
{g}\left(  b\right)  -\varphi\right]  ,\ $set $\omega_{A}^{n,m}=b-u,$ and
define $\alpha_{A}\left(  \theta_{m}^{n}\right)  :=\omega_{A}^{n,m}.$ Then
$\left[  \tilde{g}\left(  \omega_{A}^{n,m}\right)  -\varphi\right]  =\left[
\tilde{g}\left(  b-u\right)  -\varphi\right]  =\left[  \tilde{g}\left(
b\right)  -\varphi\right]  -\left[  \tilde{g}\left(  u\right)  \right]
=\left[  0\right]  .$ Choose a cochain $g_{m}^{n}\in Hom^{2-m-n}$ $\left(
A^{\otimes m},B^{\otimes n}\right)  $ such that
$
\nabla g_{m}^{n}=g^{\otimes n}\omega_{A}^{n,m}-\varphi,
$
and define $\beta\left(  \mathfrak{f}_{m}^{n}\right)  :=g_{m}^{n}.$ To extend
$\beta$ as a map of relative matrads, define\linebreak $\beta\big((
\mathfrak{f}_{1}^{1})^{\otimes n}\,\theta_{m}^{n}\big):=g^{\otimes n}\,
\omega_{A}^{n,m}.$ Passing to the limit we obtain the desired maps $\alpha
_{A}$ and $\beta.$ $\ \square$

\subsection{Homotopy Gerstenhaber Algebras}

When a $1$-connected dga $\left(  A,d,\cdot\right)  $ over a field
$\mathbf{k}$ admits a hGa structure, it lifts to the bar construction $BA$
and  induces a Hopf compatible product $\mu_{BA}$ so that $BA$ is a dgHa.
Furthermore, the dgHa structure$\ $on $BA$ lifts to a gHa structure on
$H^{\ast}\left(  BA;\mathbf{k}\right)  $. Since such liftings are required in
the application below, we include a brief review of hGa's for completeness.
To  avoid sign complications, we limit our discussion to $\mathbb{Z}_{2}%
$-dga's  and follow the exposition given by Kadeishvili in \cite{Kade}; for a
general  exposition see \cite{GV}.

A (not necessarily $1$-connected or commutative) $\mathbb{Z}_{2}$-dga $\left(
A,d,\cdot\right)  $ is a \emph{homotopy Gerstenhaber algebra} (hGa)\emph{\ }%
if  there exist multilinear operations
\[
E:=\left\{  E_{0,1}=E_{1,0}=\mathbf{1}_{A}\right\}  \cup\{E_{1,q}:A\otimes
A^{\otimes q}\rightarrow A\}_{q\geq1}
\]
such that $|E_{1,q}|=-q,$ and satisfy the following relations:

\bigskip

\noindent$dE_{1,q}\left(  a;b_{1},\ldots,b_{q}\right)  +E_{1,q}\left(
da;b_{1},\ldots,b_{q}\right)  +\sum\nolimits_{i}E_{1,q}\left(  a;b_{1}
,\ldots,db_{i},\ldots,b_{q}\right)  $
\[
=b_{1}\cdot E_{1,q-1}\left(  a;b_{2},\ldots b_{q}\right)  +E_{1,q-1}\left(
a;b_{1},\ldots,b_{q-1}\right)  \cdot b_{q}
\]
\begin{equation}
+\sum\nolimits_{i}E_{1,q-1}\left(  a;b_{1},\ldots,b_{i}\cdot b_{i+1}
,\ldots,b_{q}\right)  \label{one}%
\end{equation}
\medskip

\noindent$E_{1,q}\left(  a_{1}\cdot a_{2};b_{2},\ldots b_{q}\right)
=a_{1}\cdot E_{1,q}\left(  a_{2};b_{1},\ldots b_{q}\right)  +E_{1,q}\left(
a_{1};b_{1},\ldots b_{q}\right)  \cdot a_{2}\vspace{-0.1in}$%

\begin{equation}
\hspace*{0.75in}+\sum\nolimits_{p=1}^{q-1}E_{1,p}\left(  a_{1};b_{1},\ldots
b_{p}\right)  \cdot E_{1,q-p}\left(  a_{2};b_{p+1},\ldots b_{q}\right)
\label{two}%
\end{equation}
\medskip

\noindent$E_{1,n}\left(  E_{1,m}\left(  a;b_{2},\ldots b_{m};c_{1}
,\ldots,c_{n}\right)  \right)  =\sum\nolimits_{0\leq i_{1}\leq j_{1}\leq
\cdots\leq i_{m}\leq j_{m}\leq n}$
\[
E_{1},_{m+n+\left(  i_{1}+\cdots+i_{m}\right)  -\left(  j_{1}+\cdots
+j_{m}\right)  }(a;c_{1},\ldots,c_{i_{1}},E_{1},_{j_{1}-i_{1}}\left(
b_{1};c_{i_{1}+1},\ldots,c_{j_{1}}\right)  ,
\]
\[
c_{j_{1}+1},\ldots,c_{i_{2}},E_{1},_{j_{2}-i_{2}}\left(  b_{2};c_{i_{2}
+1},\ldots,c_{j_{2}}\right)  ,c_{j_{2}+1},\ldots,c_{i_{m}},
\]
\begin{equation}
E_{1},_{j_{m}-i_{m}}\left(  b_{m};c_{i_{m}+1},\ldots,c_{j_{m}}\right)
,c_{j_{m}+1},\ldots,c_{n}). \label{three}%
\end{equation}
\medskip

\noindent Denote $E_{1,1}$ by $\smile_{1};$ setting $q=1,$ relations
(\ref{one}) and (\ref{two}) reduce to
\[
d\left(  a\smile_{1}b\right)  +da\smile_{1}b+a\smile_{1}db=a\cdot b+b\cdot a
\text{ \ and}
\]
\[
\left(  a\cdot b\right)  \smile_{1}c=a\cdot(b\smile_{1}c)+\left(  a\smile
_{1}c\right)  \cdot b.
\]
Thus $\smile_{1}$ measures the deviation of $\cdot$ from commutativity and is
a right derivation of the product. Setting $q=2,$ relation (\ref{one})
reduces  to
\[
dE_{1,2}\left(  a;b,c\right)  +E_{1,2}\left(  da;b,c\right)  +E_{1,2}\left(
a;db,c\right)  +E_{1,2}\left(  a;b,dc\right)
\]
\[
=a\smile_{1}\left(  b\cdot c\right)  +\left(  a\smile_{1}b\right)  \cdot
c+b\cdot\left(  a\smile_{1}c\right)  .
\]
Thus $\smile_{1}$ is a left derivation up to homotopy.

Let $\left(  A,d,\cdot\right)  $ be a $1$-connected dga with an hGa
structure $E$. Consider
the tensor coalgebra $BA\otimes BA$ with coproduct $\psi:=\sigma_{2,2}\left(
\Delta_{BA}\otimes\Delta_{BA}\right)  .$ Define $\psi^{\left(  0\right)  }
:=$\textbf{$1$} and $\psi^{\left(  k\right)  }:=\left(  \psi\otimes
\mathbf{1}^{\otimes k-1}\right)  \cdots\left(  \psi\otimes\mathbf{1}\right)
\psi,$ where $\mathbf{1}$ denotes the identity on $BA\otimes BA.$
Comultiplicatively extend the hGa structure maps $E_{0,1}=E_{1,0}
=\mathbf{1}_{A}$ as coalgebra maps $E_{0,1}:\left[  \ \right]  \mathbf{\otimes
}BA\rightarrow BA$ and $E_{1,0}:BA\otimes\left[  \ \right]  \rightarrow BA.$
Then $E_{0,1}$ and $E_{1,0}$ have degree zero, are undefined except with
respect to units, i.e., $E_{0,1}\left(  [\ ]\otimes\lbrack x]\right)
=E_{1,0}\left(  [x]\otimes\lbrack\ ]\right)  =[x]$, and generate the shuffle
product
\[
sh:=\sum_{k\geq1}\left(  E_{0,1}+E_{1,0}\right)  ^{\otimes k}\psi^{\left(
k-1\right)  }:BA\otimes BA\rightarrow BA.
\]
For example, $sh\left(  \left[  a|b\right]  \otimes\left[  c\right]  \right)
=\left(  E_{0,1}+E_{1,0}\right)  ^{\otimes3}\psi^{\left(  2\right)  }\left(
\left[  a|b\right]  \otimes\left[  c\right]  \right)  =\left[  a|b|c\right]
+\left[  a|c|b\right]  +[c|a|b].$

In general, the dgHa structure of $\left(  BA,d_{BA},\Delta_{BA},sh\right)  $
fails to induce a gHa structure on $H=H^{\ast}\left(  BA\right)  $. However,
an induced gHa structure $\left(  H,\Delta,\mu\right)  $ is obtained by
comultiplicatively extending the hGa structure and perturbing the shuffle
product, i.e.,
\[
\mu_{BA}:=\sum_{k\geq1}\left(  E_{0,1}+E_{1,0}+E_{1,1}+E_{1,2}+\cdots\right)
^{\otimes k}\psi^{\left(  k-1\right)  }:BA\otimes BA\rightarrow BA.
\]
Then for example, $\mu_{BA}\left(  \left[  a\right]  \otimes\left[  b\right]
\right)  =\left[  a|b\right]  +[b|a]+\left[  a\smile_{1}b\right]  ,$ and in
particular, $\mu_{BA}\left(  \left[  a\right]  \otimes\left[  a\right]
\right)  =\left[  a\smile_{1}a\right]  $.

\subsection{A Non-trivial G-S Extension of Loop Cohomology}

Let $Y:=\left(  S^{2}\times S^{3}\right)  \vee\Sigma\mathbb{C}P^{2}$ and
consider the total space $X$ of the $2$-stage Postnikov system
\[
\begin{array}
[c]{ccccc}%
K\left(  \mathbb{Z}_{2},4\right)  & \longrightarrow & X &  & \\
&  & \downarrow &  & \\
&  & Y & \overset{f}{\longrightarrow} & K\left(  \mathbb{Z}_{2},5\right)
\text{ .}\\
&  & a_{2}a_{3}+Sq^{2}b & \underset{f^{\ast}}{\longleftarrow} & \iota_{5}%
\end{array}
\]
Denote the generators of $A:=H^{\ast}\left(  X;\mathbb{Z}_{2}\right)  $ by
$a_{i}\in H^{i}\left(  S^{i};\mathbb{Z}_{2}\right)  ,$ $\left\{
b,Sq^{2}b\right\}  \in$\linebreak$H^{\ast}(\Sigma\mathbb{C}P^{2}
;\mathbb{Z}_{2}),$ and $\left\{  Sq^{1}\iota_{4},Sq^{2}\iota_{4}
,\ldots\right\}  \in H^{\ast}\left(  \mathbb{Z}_{2},4;\mathbb{Z}_{2}\right)
.$ The hGa structure of $A$ is non-degenerate with $E_{1,1}:A\otimes
A\rightarrow A$ given by
\[
E_{1,1}\left(  b\otimes b\right)  =Sq^{2}b=a_{2}a_{3}.
\]
The bar construction $BA$ with standard differential $d$ and cofree coproduct
$\Delta_{BA}$ is a dg coalgebra. Note that $d(\left[  a_{2}|a_{3}\right]
+\left[  a_{3}|a_{2}\right]  )=0.$ Lift $E_{1,0},$ $E_{0,1},$ and $E_{1,1}$
to  $BA$ and extend as coalgebra maps. Then $\mu_{BA}$ acts as the shuffle
product  except
\[
\mu_{BA}\left(  \left[  b\right]  \otimes\left[  b\right]  \right)  =\left[
a_{2}a_{3}\right]  =d\left[  a_{2}|a_{3}\right]  ,
\]
$(BA,d,\Delta_{BA},\mu_{BA})$ is a dgHa, and $H:=H^{\ast}\left(
BA;\mathbb{Z}_{2}\right)  \approx H^{\ast}\left(  \Omega X;\mathbb{Z}
_{2}\right)  $ as modules.

Let $\alpha_{i-1}:=cls[a_{i}],\ \beta:=cls[b],$ and $\gamma:=cls([a_{2}
|a_{3}]+[a_{3}|a_{2}]);$ then the induced product and coproduct $\mu$ and
$\Delta$ on $H$ act as in Example \ref{Ex1} so that $\left(  H,\mu,\Delta\right)
$ is a gHa. Represent $\gamma$ by $\bar{\gamma}:=\left[  a_{2}|a_{3}\right]
+\left[  a_{3}|a_{2}\right]  ,$ a generator $x\neq\gamma$ by $\bar{x}:=\left[
\uparrow x\right]  $, and a general class $y_{1}|\cdots|y_{n}$ by $\bar{y}
_{1}|\cdots|\bar{y}_{n}$. Define a cocycle-selecting homomorphism
$g:H\rightarrow BA$ by $g\left(  y_{1}|\cdots|y_{n}\right) : =\bar{y}
_{1}|\cdots|\bar{y}_{n}$; then the Transfer Algorithm transfers the dgHa
structure on $BA$ to an $A_{\infty}$-bialgebra structure on $H$ along $g$,
which specializes to a strict $A_{k}$-bialgebra structure for each $k\geq3$.

S. Saneblidze was the first to consider hGa's with non-trivial actions of the
Steenrod algebra $\mathcal{A}_{2}$ in \cite{Sane}. In general, the Steenrod
$\smile_{1}$-cochain operation together with other higher cochain operations
induce a non-trivial hGa structure on $S^{\ast}\left(  X;\mathbb{Z}
_{2}\right)  $, but the failure of the differential to be a $\smile_{1}
$-derivation prevents an immediate lifting of the hGa structure to cohomology
(for some remarks on the history of lifting a $\smile_{1}$-operation on
homology see \cite{KS} and \cite{Sane}).

When no multiplicative map $A\rightarrow C$ of dga's exists, as is the case
when $A=BH^{\ast}(X;\mathbb{Z}_{2})$ and $C=S^{\ast}(\Omega X;\mathbb{Z}_{2}
)$, there may exist a family of dga's $\{B_{i}\}$ and a zig-zag of
multiplicative maps $A\leftarrow B_{1}\cdots B_{k}\rightarrow C$. Indeed, in
our application we have $BH^{\ast}(X;\mathbb{Z}_{2})\leftarrow B(RH^{\ast
}(X;\mathbb{Z}_{2}))\leftarrow B(R_{a}H^{\ast}(X;\mathbb{Z}_{2}))\rightarrow
B(S^{\ast}(X;\mathbb{Z}_{2})\rightarrow S^{\ast}(\Omega X;\mathbb{Z}_{2})$,
where the first is induced by the Hirsch resolution map $H^{\ast}
(X;\mathbb{Z}_{2})\leftarrow RH^{\ast}(X;\mathbb{Z}_{2})$, the second is
induced by the Hirsch resolution projection $RA\leftarrow R_{a}H^{\ast
}(X;\mathbb{Z}_{2})$, where $R_{a}H^{\ast}(X)$ denotes the Hirsch
(absolute) resolution of $H^{\ast}(X)$, the third is induced
by the Hirsch modeling map $R_{a}H^{\ast}(X;\mathbb{Z}_{2})\rightarrow
S^{\ast}(X;\mathbb{Z}_{2})$, and the fourth is standard. Under this
zig-zag, $H:=H^{\ast}\left(  BA;\mathbb{Z}_{2}\right)  $ is a gHa model for
$H^{\ast}\left(  \Omega X;\mathbb{Z}_{2}\right)  .$

\begin{proposition}
\label{opers}The gHa model $H\approx H^{\ast}(\Omega(X);\mathbb{Z})$ admits a topologically invariant induced G-S bialgebra structure $\{\omega^1_3
,\omega^2_2,\omega^3_1\}$ such that
\[
\omega^1_3\neq 0,\text{ }\omega^2_2\neq0,\text{ and }\omega^3_1\equiv 0.
\]
Thus $\left(  H,\mu,\Delta,\omega^2_2,\omega^1_3\right)  $ is a G-S
extension of $H$. 
\end{proposition}

\noindent\textbf{Proof.} First, by the Transfer Algorithm Theorem, there is a
cochain homotopy $g_{1}^{2}:H\rightarrow BA\otimes BA$ satisfying the
$JJ^2_1$ structure relation $\nabla g_{1}^{2}=\Delta_{BA}g+(g\otimes
g)\Delta.$ Since $\nabla g_{1}^{2}=0$ by the comultiplicativity of $g,$ we
may  choose $g_{1}^{2}=0.$ Dually, note that
\[
\left(  g\mu+\mu_{BA}(g\otimes g)\right)  (x\otimes y)=\left\{
\begin{array}
[c]{cc}%
\left[  a_{2}a_{3}\right]  , & x\otimes y=\beta\otimes\beta\\
0, & \text{otherwise.}%
\end{array}
\right.
\]
By the Transfer Algorithm, there is a cochain homotopy $g_{2}^{1}:H\otimes
H\rightarrow BA$ satisfying the $JJ^1_2$ structure relation $\nabla
g_{2}^{1}=g\mu+\mu_{BA}(g\otimes g)$ such that for some $i\in\left\{
2,3\right\}  $
\[
g_{2}^{1}(x\otimes y)=\left\{
\begin{array}
[c]{cc}%
\left[  a_{i}|a_{5-i}\right]  , & x\otimes y=\beta\otimes\beta\\
0, & \text{otherwise.}%
\end{array}
\right.
\]
Choose $i=2$ so that $g_{2}^{1}(\beta\otimes\beta)=\left[  a_{2}|a_{3}\right]
$ (the choice $i=3$ gives rise to an isomorphic structure); the analysis in
\cite{Umble2} implies
\[
\omega^2_2\left(  \beta\otimes\beta\right)  =\alpha_{1}\otimes\alpha_{2}.
\]

Second, by the Transfer Algorithm, there is a cochain homotopy $g_{1}
^{3}:H\rightarrow BA^{\otimes3}$ satisfying the $JJ^3_1$ structure relation
\begin{equation}
\nabla g_{1}^{3}=g^{\otimes3}\omega^3_1+\left(  g\otimes g_{1}^{2}+g_{1}
^{2}\otimes g\right)  \Delta+\left(  \Delta_{BA}\otimes\mathbf{1}
+\mathbf{1}\otimes\Delta_{BA}\right)  g_{1}^{2}+\omega_{BA}^{3,1}g.
\label{JJ31}%
\end{equation}
Since $\omega_{BA}^{3,1}=0$ and $g_{1}^{2}=0$ by the choice above,
(\ref{JJ31}) reduces to $\nabla g_{1}^{3}=g^{\otimes3}\omega^3_1=\tilde
{g}\left(  \omega^3_1\right)  $ and vanishes in cohomology. Since $H$ is
free as a $\mathbb{Z}_{2}$-module, $\tilde{g}:Hom^{\ast}(H,$ $H^{\otimes3})
\rightarrow Hom^{\ast}\left(  H,BA^{\otimes3}\right)  $ is a cohomology
isomorphism and it follows that
\[
\omega^3_1\equiv 0.
\]

Dually, there is a cochain homotopy $g_{3}^{1}:H^{\otimes3}\rightarrow BA$
satisfying the $JJ^1_3$ structure relation
\begin{equation}
\nabla g_{3}^{1}=g\omega^1_3+\mu_{BA}\left(  g\otimes g_{2}^{1}+g_{2}
^{1}\otimes g\right)  +g_{2}^{1}\left(  \mu\otimes\mathbf{1}+\mathbf{1}
\otimes\mu\right)  +\omega_{BA}^{1,3}g^{\otimes3}. \label{JJ13}%
\end{equation}
For simplicity let $\phi_{3}^{1}:=\mu_{BA}\left(  g\otimes g_{2}^{1}+g_{2}
^{1}\otimes g\right)  +g_{2}^{1}\left(  \mu\otimes\mathbf{1}+\mathbf{1}
\otimes\mu\right)  $ and note that
\[
\phi_{3}^{1}\left(  \beta\otimes\beta\otimes\sigma\right)  =\phi_{3}
^{1}\left(  \sigma\otimes\beta\otimes\beta\right)  =\left\{
\begin{array}
[c]{cc}%
\mu_{BA}\left(  \left[  a_{2}|a_{3}\right]  \otimes\bar{\sigma}\right)  , &
\sigma\neq1,\beta\\
0, & \text{otherwise.}%
\end{array}
\right.
\]
Since $\omega_{BA}^{1,3}=0,$ it follows that $\nabla g_{3}^{1}=g\omega
^1_3+\phi_{3}^{1}$ so that $g\omega^1_3$ and $\phi_{3}^{1}$ are
cohomologous in $Hom\left(  H^{\otimes3},BA\right)  $. Since $\tilde{g}$ is a
cohomology isomorphism, it follows that
\[
\omega^1_3\left(  \beta\otimes\beta\otimes\sigma\right)  =\omega
^1_3\left(  \sigma\otimes\beta\otimes\beta\right)  =\left\{
\begin{array}
[c]{cc}%
\mu\left(  \alpha_{1}|\alpha_{2}\otimes\sigma\right)  , & \sigma\neq1,\beta\\
0, & \text{otherwise.}%
\end{array}
\right.
\]

Finally, $\omega$ is invariant by uniqueness in the Transfer Theorem.
$\ \square$

\begin{proposition}
\label{prop}The G-S extension in Proposition \ref{opers} is non-trivial. 
\end{proposition}

\noindent\textbf{Proof.} By Theorem \ref{four}, Part 1, the cochain
$\omega^2_2+\omega^1_3\in Z_{GS}^{2,-1}\left(  H;H\right)  $. I claim
$cls(\omega^2_2+\omega^1_3)\neq0$.

Suppose $f=f_{1}^{2}\in Hom^{-1}(H,H^{\overline{\otimes}2})$ satisfies
$\partial f=\omega^2_2.$ Since $\omega^2_2\left(  \beta\otimes
\beta\right)  =\alpha_{1}\otimes\alpha_{2}\ $by the proof of Proposition
\ref{opers}, evaluating at $\beta\otimes\beta$ gives
\[
(\mu\otimes\mu)\sigma_{2,2}\big(\Delta(\beta)\otimes
f(\beta)+f(\beta)\otimes\Delta(\beta)\big)=\alpha_{1}\otimes\alpha_{2}.
\]
Since each component of the left-hand side 
has a factor involving $\beta$, this is impossible.

Suppose $f=f_{2}^{1}\in Hom^{-1}(H^{\otimes\underline{2}},H)$ satisfies
$\delta f=\omega^2_2;$ then
\[
(\mu\otimes f+f\otimes\mu)\sigma_{2,2}(\Delta\otimes\Delta
)+\Delta f=\omega^2_2.
\]
Note that $\mu\left(  \beta\otimes\beta\right)  =0$ since $\mu$ acts as the
shuffle product, and $f\left(  1\otimes1\right)  =0$ for dimensional reasons.
Evaluating the left-hand-side at $\beta\otimes \beta$ we obtain
\[
1\otimes f\left(  \beta\otimes\beta\right)  +\beta\otimes f\left(
1\otimes\beta\right)  +\beta\otimes f\left(  \beta\otimes1\right).
\]
Thus the required (primitive) component $f\left( \beta\otimes\beta\right)  \otimes1$ 
is missing from
\[
\Delta f\left(  \beta\otimes\beta\right)  =1\otimes f\left(  \beta\otimes
\beta\right)  +\beta\otimes f\left(  1\otimes\beta\right)  +\beta\otimes
f\left(  \beta\otimes1\right)  +\alpha_{1}\otimes\alpha_{2},
\]
and this too is impossible.

Therefore $
D(f_{1}^{2}+f_{2}^{1})\neq\omega^2_2+\omega^1_3$\
for all $f_{1}^{2}\in Hom^{-1}(H,H^{\overline{\otimes}2})$ and all $f_{2}
^{1}\in Hom^{-1}(H^{\otimes\underline{2}},H),$ and it follows that
$cls(\omega^2_2+\omega^1_3)\neq0$ as claimed. The conclusion follows by
Theorem \ref{four}, Part 2. $\ \square$\medskip

There is a family of spaces $\{X_k\}_{k\geq 3}$ whose loop cohomology 
$H^{\ast}(\Omega X_{k};\mathbb{Z}_{2})$ admits an induced topologically invariant 
$A_{\infty}$-bialgebra structure $\{\omega_k^{j,i}\}_{i+j>3}$  such that 
$\omega_k^{k,2}\neq0$ for each $k\geq2$  
(see  Example 12.5 in \cite{SU2}).
Unfortunately, when $k\geq3$ the required $KK$ structure relations cannot be
expressed in terms of the G-S differentials, and Theorem \ref{four} does
not  apply. We would welcome an enrichment of the G-S complex that supports
such  expressions. One approach might be to extend the G-S complex to a
multicomplex with additional differentials defined in terms of the higher
order operations. 

Finally, the dgHa model $H\approx H^*(\Omega X;\mathbb{Z}_2)$ in our application admits an
$A_k$-bialgebra structure for each $k\geq 3$.  It would be nice to have a family 
of spaces $\{X_k\}_{k\ge 3}$ such that 
$H^*(X_k)$ admits an $A_k$ but not an $A_{k+1}$-bialgebra structure. 
We leave this problem for the reader.\medskip

\noindent\textbf{Acknowledgement. }I wish to thank Jim Stasheff for his
thoughtful comments and probing questions.

\end{document}